\documentclass[10pt]{article}
\usepackage{graphicx}
\usepackage{amssymb}
\usepackage{color,graphicx}
\usepackage{amsmath,amsthm,amscd}
\usepackage[latin1]{inputenc}
\usepackage{fancyhdr}
\usepackage{pb-diagram}
\numberwithin{equation}{section}
\newtheorem{theorem}{Theorem}[section]

\newtheorem{proposition}[theorem]{Proposition}

\newtheorem{definition}[theorem]{Definition}
\theoremstyle{definition}
\newtheorem{example}[theorem]{Example}

\newtheorem{remark}[theorem]{Remark}

\newfont{\fuentea}{cmsy10 at 16pt}
\newfont{\fuenteb}{cmsy10 at 10pt}
\DeclareFixedFont{\zcal}{OT1}{pzc}{m}{sl}{16pt}%
\DeclareFixedFont{\titulo}{OT1}{pbk}{bx}{sc}{30pt}%
\DeclareFixedFont{\trm}{OT1}{ptm}{b}{sc}{12pt}%
\DeclareFixedFont{\trw}{OT1}{phv}{b}{bx}{14pt}%

\def\@roman#1{\romannumeral #1}
\makeatother

\title{\textbf{Elementary matrix-computational proof\\ of Quillen-Suslin theorem for Ore extensions}}
\author{William Fajardo\footnote{The first author was supported by
Institución Universitaria Politécnico Grancolombiano}\\ Oswaldo Lezama\footnote{The second author
was supported by the project \textit{New trends of non-commutative algebra and skew PBW
extensions}, HERMES CODE 26872,
Universidad Nacional de Colombia.} \\
Seminario de Álgebra Constructiva - $\text{SAC}^2$\\
Departamento de Matemáticas\\
Universidad Nacional de Colombia, Bogotá, COLOMBIA\\
\texttt{jolezamas@unal.edu.co}\\
}
\date{}
\begin{document}
\maketitle
\begin{abstract}
\noindent In this short note we present an elementary matrix-constructive proof of Quillen-Suslin
theorem for Ore extensions: If $K$ is a division ring and $A:=K[x;\sigma,\delta]$ is an Ore
extension, with $\sigma$ bijective, then every finitely generated projective $A$-module is free. We
will show an algorithm that computes the basis of a given finitely generated projective module. The
algorithm has been implemented in a computational package, and some illustrative examples are
included.

\bigskip


\bigskip

\end{abstract}
\section{Introduction}

\noindent When a new type of ring is defined, it is an interesting problem to investigate if the
finitely generated projective modules over it are free. This problem becomes classical after the
formulation in 1955 of the famous Serre's problem about the freeness of finitely generated
projective modules over the polynomial ring $K[x_1,\dots,x_n]$, $K$ a field (see \cite{Artamonov2},
\cite{Artamonov3}, \cite{Bass1}, \cite{Lam}). The Serre's problem was solved positively, and
independently, by Quillen in USA, and by Suslin in Leningrad, USSR (St. Petersburg, Russia) in 1976
(\cite{Quillen}, \cite{Suslin1}).

\begin{definition}\label{PFrings}
Let $S$ be a ring. $S$ is a $\mathcal{PF}$ ring if every finitely generated $($f.g.$)$ projective
$S$-module is free.
\end{definition}

\begin{theorem}[Quillen-Suslin; \cite{Quillen}, \cite{Suslin1}]
$K[x_1,\dots,x_n]$ is $\mathcal{PF}$.
\end{theorem}

The goal of this short paper is to present an elementary matrix-constructive proof of
Quillen-Suslin theorem for single Ore extensions over division rings, i.e, if $K$ is a division
ring and $A:=K[x;\sigma,\delta]$ is an Ore extension, with $\sigma$ a bijective endomorphism of $K$
and $\delta$ a $\sigma$-derivation, then $A$ is $\mathcal{PF}$. Our proof is supported in a matrix
characterization of $\mathcal{PF}$ rings given in \cite{Gallego}.

\begin{proposition}[\cite{Gallego}]\label{6.2.4}
Let $S$ be a ring. $S$ is $PF$ if and only if for every $s\geq 1$, given an idempotent matrix $F\in
M_s(S)$, there exists a matrix $U\in GL_s(S)$ such that
\begin{equation}\label{eq6.2.4}
UFU^{-1}=\begin{bmatrix}0 & 0\\
0 & I_r \end{bmatrix},
\end{equation}
where $r=dim(\langle F\rangle)$, $0\leq r\leq s$, and $\langle F\rangle$ represents the left
$S$-module generated by the rows of $F$. Moreover, a basis of $M$ is given by the last $r$ rows of
$U$.
\end{proposition}

\section{Quillen-Suslin theorem: Elementary matrix proof}

\noindent In this section we will prove that the Ore extension $K[x;\sigma,\delta]$ is
$\mathcal{PF}$. Despite of this fact is well-known (see \cite{Cohn1}), our proof is elementary and
matrix-constructive, and allow to exhibit an algorithm that computes the basis of a given finitely
generated projective modules.

\begin{theorem}[Quillen-Suslin]\label{theorem9.6.1}
Let $K$ be a division ring and $A:=K[x;\sigma,\delta]$, with $\sigma$ bijective. Then $A$ is
$\mathcal{PF}$.
\end{theorem}
\begin{proof}
Let $s\geq 1$ and let $F=[f_{ij}]\in M_s(A)$ be an idempotent matrix, the proof is by induction on
$s$ and we will follow a procedure as in Proposition 64 of \cite{Gallego}. We will use the
relations that satisfy the entries of $F$, in particular, the following two relations:
\begin{align*}
& f_{11}^2+f_{12}f_{21}+f_{13}f_{31}+\cdots+f_{1s}f_{s1}=f_{11},\\
& f_{11}f_{12}+f_{12}f_{22}+f_{13}f_{32}+\cdots+f_{1s}f_{s2}=f_{12}.
\end{align*}

$\textbf{s=1}$: In this case $F=[f]$; since $A$ is a domain, its idempotents are trivial, then
$f=1$ or $f=0$ and hence $U=[1]$.

$\textbf{s}\geq \textbf{2}$: Now suppose that the result holds for $s-1$ and let $F=[f_{ij}]\in
M_s(A)$ be an idempotent matrix. We have two possibilities.

(A) All elements in the first row and in the first column of $F$ are zero. Then we apply induction.

(B) Suppose that there exists at least one non zero element in the first row (the reasoning for the
first column is similar); we can assume that this element is $f_{11}$ (if $f_{11}=0$ and
$f_{1j}\neq 0$ then we can change $F$ by $TFT^{-1}$ with $T:=I_s-E_{j1}$). Then arise two
possibilities.

(B1) $\deg(f_{11})=0$, so $f_{11}\in K-0$, i.e., $f_{11}$ is invertible. Then taking

\[U:=\begin{bmatrix} 1& f_{11}^{-1}f_{12}& f_{11}^{-1}f_{13} &\cdots &f_{11}^{-1}f_{1s}\\ -f_{21}f_{11}^{-1} & 1 & 0 &\cdots & 0\\
 -f_{31}f_{11}^{-1} & 0 & 1&\cdots & 0\\ \vdots & & &\cdots & \\-f_{s1}f_{11}^{-1} & 0 & 0& \cdots & 1\end{bmatrix}\]
we have that $U\in GL_{s}(A)$ and its inverse is
\[U^{-1}=\begin{bmatrix} f_{11}& -f_{12}& -f_{13} &\cdots &-f_{1s}\\ f_{21} & -f_{21}f_{11}^{-1}f_{12}+1 & -f_{21}f_{11}^{-1}f_{13}&\cdots
& -f_{21}f_{11}^{-1}f_{1s}\\
 f_{31} & -f_{31}f_{11}^{-1}f_{12} & -f_{31}f_{11}^{-1}f_{13}+1&\cdots & -f_{31}f_{11}^{-1}f_{1s}\\ \vdots & & &\cdots & \\f_{s1}
& -f_{s1}f_{11}^{-1}f_{12} & -f_{s1}f_{11}^{-1}f_{13}& \cdots &
-f_{s1}f_{11}^{-1}f_{1s}+1\end{bmatrix}.\] Moreover, $UFU^{-1}=\begin{bmatrix}1 &
0_{1,s-1}\\0_{s-1,1}& F_{1}\end{bmatrix}$, where $F_{1}\in M_{s-1}(A)$ is an idempotent matrix,
therefore we can apply induction.

(B2) $\deg(f_{11}):=n\geq 1$; since $A$ is a domain at least one non diagonal entry in the first
row and in the first column of $F$ are non zero: In fact, if $f_{12}=\cdots=f_{1s}=0$, then
$f_{11}=1$ or $f_{11}=0$, false; similarly if $f_{21}=\cdots=f_{s1}=0$. Using elementary and
permutation matrices, no affecting the entry $f_{11}$, we can reduce the degrees of $f_{12},
\dots,f_{1s}$ until the situation in which $f_{12}\neq 0$ and $f_{13}=\cdots =f_{1s}=0$ (a similar
reasoning apply for the first column); then we have $f_{11}^2+f_{12}f_{21}=f_{11}$ and $f_{21}\neq
0$; note that $\deg(f_{11}^2)=2n$, so $\deg(f_{21}):=p\leq n$ or $\deg(f_{12}):=q\leq n$; let
$a_n:=lc(f_{11})$, $c_p:=lc(f_{21})$ and $b_q:=lc(f_{12})$.

If $p\leq n$ then
\[
TFT^{-1}=F'=\begin{bmatrix}f_{11}'& f_{12}' & f_{13}' & \dots & f_{1s}'\\f_{21} & f_{22}' & f_{23} & \dots & f_{2s}\\
\vdots & \vdots & \vdots & \dots & \vdots\\
f_{s1} & f_{s2}' & f_{s3} & \dots & f_{ss}\end{bmatrix},
\]
with $T:=I_s-a_n\sigma^{n-p}(c_p^{-1})x^{n-p}E_{12}$; note that $F'$ is idempotent; moreover
$f_{11}'=0$ or $f_{11}'\neq 0$; if $f_{11}'\neq 0$ then arise two options: $\deg(f_{11}')=0$, i.e.,
$f_{11}'\in K-0$ or $1\leq \deg(f_{11}')\leq n-1$ and again $\deg(f_{21})\leq \deg{f_{11}'}$ or
$\deg(f_{12}')\leq \deg{f_{11}'}$.

If $p>n$ but $q\leq n$ then
\[
LFL^{-1}=F''=\begin{bmatrix}f_{11}''& f_{12} & f_{13} & \dots & f_{1s}\\f_{21}'' & f_{22}'' & f_{23}'' & \dots & f_{2s}''\\
\vdots & \vdots & \vdots & \dots & \vdots\\
f_{s1}'' & f_{s2} & f_{s3} & \dots & f_{ss}\end{bmatrix},
\]
with $L:=I_s+\sigma^{-q}(b_q^{-1}a_n)x^{n-q}E_{21}$; note that $F''$ is idempotent; moreover
$f_{11}''=0$ or $f_{11}''\neq 0$; if $f_{11}''\neq 0$ then arise two options: $\deg(f_{11}'')=0$,
i.e., $f_{11}''\in K-0$ or $1\leq \deg(f_{11}'')\leq n-1$ and again $\deg(f_{12})\leq
\deg{f_{11}''}$ or $\deg(f_{21}'')\leq \deg{f_{11}''}$.

We can repeat this reasoning for $F'$ and $F''$ and we obtain an idempotent matrix $G=[g_{ij}]$
similar to $F$ with $g_{11}=0$ or $g_{11}\in K-0$; if $g_{11}\in K-0$ we conclude using the case
(B1). Then assume that $g_{11}=0$; if all elements in the first row and in the first column of $G$
are zero, then we can apply induction and we finish. If not, then in a similar way as was remarked
above, using elementary and permutation matrices, no affecting the first column, in particular the
entry $g_{11}$, we can reduce the degrees of $g_{12},\dots, g_{1s}$ until the situation in which
$g_{12}\neq 0$ and $g_{13}=\dots=g_{1s}=0$ (a similar reasoning apply for the first column); thus,
from $g_{12}g_{22}=g_{12}$ we obtain that $g_{22}=1$ and hence by the permutation matrix $P_{12}$
we finish using the case (B1).
\end{proof}

\section{The algorithm}

\noindent In this section we present the algorithm for computing the matrix $U$ in the proof of
Quillen-Suslin theorem for Ore extensions (Theorem \ref{theorem9.6.1}); the algorithm also
calculates the basis of a given finitely generated projective module (Proposition \ref{6.2.4}). We
present two versions of the algorithm, a constructive simplified version, and a more complete
computational version over fields. The computational version was implemented using ${\rm
Maple}^\circledR$ 2016 (see Remark \ref{remark4.2} below).

\begin{center}
\fbox{\parbox[c]{12cm}{
\begin{center}
{\rm \textbf{Algorithm for the Quillen-Suslin theorem:\\ Constructive version}}
\end{center}
\begin{description}
\item[]{\rm \textbf{INPUT}:} An Ore extension $A:=K[x,\sigma,\delta]$ ($K$ a division ring, $\sigma$ bijective); $F\in M_s(A)$ an idempotent
matrix.
\item[]{\rm \textbf{OUTPUT}:} Matrices $U$, $U^{-1}$ and a
basis $X$ of $\langle F\rangle$, where
\begin{equation}\label{equation14.6.1}
UFU^{-1}=\begin{bmatrix}0 & 0\\
0 & I_r \end{bmatrix} \ \text{and}\ r=dim(\langle F\rangle).
\end{equation}
\item[]{\rm \textbf{INITIALIZATION}:} $F_1:=F$.

{\rm \textbf{FOR}} $k$ from $1$ to $n-1$ {\rm \textbf{DO}}
\begin{enumerate}
\item Follow the reduction procedures
(B1) and (B2) in the proof of Theorem \ref{theorem9.6.1} in order to compute matrices $U_k'$,
$U_k'^{-1}$ and $F_{k+1}$ such that
\begin{equation*}
U_k'F_kU_k'^{-1}=\begin{bmatrix}\alpha_k & 0\\
0 & F_{k+1} \end{bmatrix},\text{ where } \alpha_k\in\{0,1\}.
\end{equation*}
\item $U_{k}:=\begin{bmatrix}I_{k-1} & 0\\
0 & U_{k}'\end{bmatrix}U_{k-1}$; compute $U_k^{-1}$.
\item By permutation matrices modify $U_{n-1}$.
\end{enumerate}
{\rm \textbf{RETURN}} $U:=U_{n-1}$, $U^{-1}$ satisfying (\ref{equation14.6.1}), and a basis $X$ of
$\langle F\rangle$.
\end{description}}}
\end{center}

\begin{example}
For $A:=K[x,\sigma,\delta]$, with $K:=\mathbb{C}$, $\sigma(z):=\overline{z}$ and $\delta:=0$, we
consider in $M_4(A)$ the idempotent matrix {\tiny
\[
F=
\begin{bmatrix}
1-ix-x^2+(1+i)x^3 & -1+(2-i)x^2+(-1-i)x^3 & -i-x+(1+i)x^2 & 1+ix+(-1+i)x^2\\
-ix+(1+i)x^3      & ix+(1-i)x^2+(-1-i)x^3 & -i+(1+i)x^2   & 1+(-1+i)x^2\\
ix^2              & -x-ix^2               & 1+ix          & x\\
x^3-x^2           & -ix+(1-i)x^2-x^3      & x^2-x         & 1+ix+ix^2
\end{bmatrix}.
\]}
We apply the constructive version of the Quillen-Suslin algorithm, i.e., following the reductions
(B1) and (B2), we compute the matrices $U_k$ and $F_k$, for $1\leq k\leq 3$: {\tiny
\[
U_1 =
\begin{bmatrix}
1-ix-x^2+(1+i)x^3 &  -1+(2-i)x^2+(-1-i)x^3 &  -i-x+(1+i)x^2 &  1+ix+(-1+i)x^2\\
x &  -i-x &  1 &  i\\
0 &  1 &  0 &  0\\
0 &  0 &  0 &  1
\end{bmatrix},
\]}
\[
U_1^{-1} =
\begin{bmatrix}
1 &  i+x+(-1-i)x^2 &  0 &  0\\
0 &  0 &  1 &  0\\
-x &  1+ix-x^2+(1-i)x^3 &  i+x &  -i\\
0 &  0 &  0 &  1
\end{bmatrix},
\]
\[
U_1FU_1^{-1}=
\begin{bmatrix}
1 &  0 &  0 &  0\\
0 &  0 &  0 &  0\\
0 &  -i+(1+i)x^2 &  1 &  0\\
0 &  x^2-x &  0 &  1
\end{bmatrix},
F_2 =
\begin{bmatrix}
0 &  0 &  0\\
-i+(1+i)x^2 &  1 &  0\\
x^2-x &  0 &  1
\end{bmatrix};
\]
{\tiny
\[
U_2=
\begin{bmatrix}
1-ix-x^2+(1+i)x^3 &  -1+(2-i)x^2+(-1-i)x^3 &  -i-x+(1+i)x^2 &  1+ix+(-1+i)x^2\\
ix+(-1-i)x^3 &  -ix+(-1+i)x^2+(1+i)x^3 &  i+(-1-i)x^2 &  -1+(1-i)x^2\\
x &  -i-x &  1 &  i\\
0 &  0 &  0 &  1
\end{bmatrix},
\]}
\[
U_2^{-1} =
\begin{bmatrix}
1 &  0 &  i+x+(-1-i)x^2 &  0\\
0 &  -1 &  i+(-1-i)x^2 &  0\\
-x &  -i-x &  -ix^2 &  -i\\
0 &  0 &  0 &  1
\end{bmatrix},
\]
\[
U_2FU_2^{-1}=
\begin{bmatrix}
1 &  0 &  0 &  0\\
0 &  1 &  0 &  0\\
0 &  0 &  0 &  0\\
0 &  0 &  x^2-x &  1
\end{bmatrix},
F_3=
\begin{bmatrix}
0 &  0\\
x^2-x &  1
\end{bmatrix};
\]
{\tiny
\[
U_3=
\begin{bmatrix}
1-ix-x^2+(1+i)x^3 &  -1+(2-i)x^2+(-1-i)x^3 &  -i-x+(1+i)x^2 &  1+ix+(-1+i)x^2\\
ix+(-1-i)x^3 &  -ix+(-1+i)x^2+(1+i)x^3 &  i+(-1-i)x^2 &  -1+(1-i)x^2\\
-x^3+x^2 &  ix+(-1+i)x^2+x^3 &  -x^2+x &  -1-ix-ix^2\\
x &  -i-x &  1 &  i
\end{bmatrix},
\]}
\[
U_3^{-1}=
\begin{bmatrix}
1 &  0 &  0 &  i+x+(-1-i)x^2\\
0 &  -1 &  0 &  i+(-1-i)x^2\\
-x &  -i-x &  i &  -ix\\
0 &  0 &  -1 &  -x^2+x
\end{bmatrix},
\]
\[
U_3FU_3^{-1}=
\begin{bmatrix}
1 &  0 &  0 &  0\\
0 &  1 &  0 &  0\\
0 &  0 &  1 &  0\\
0 &  0 &  0 &  0
\end{bmatrix},
F_4=
\begin{bmatrix}
0
\end{bmatrix}.
\]
Finally, using permutation matrices, we get {\tiny
\[
U=
\begin{bmatrix}
x                 &                  -i-x  &             1 & i \\
1-ix-x^2+(1+i)x^3 & -1+(2-i)x^2+(-1-i)x^3  & -i-x+(1+i)x^2 & 1+ix+(-1+i)x^2\\
ix+(-1-i)x^3      & -ix+(-1+i)x^2+(1+i)x^3 & i+(-1-i)x^2   & -1+(1-i)x^2\\
-x^3+x^2          & ix+(-1+i)x^2+x^3       & -x^2+x        & -1-ix-ix^2
\end{bmatrix},
\]}

\[
U^{-1}=
\begin{bmatrix}
i+x+(-1-i)x^2 &  1 &  0   & 0\\
i+(-1-i)x^2   &  0 & -1   & 0\\
-ix           & -x & -i-x & i\\
-x^2+x        &  0 &  0   & -1
\end{bmatrix},
\]
\[
UFU^{-1}=
\begin{bmatrix}
0 & 0 & 0 & 0\\
0 & 1 & 0 & 0\\
0 & 0 & 1 & 0\\
0 & 0 & 0 & 1
\end{bmatrix}.
\]
So, $r=3$ and the last three rows of $U$ conform a basis
$X=\{\emph{\textbf{x}}_1,\textbf{\emph{x}}_2,\textbf{\emph{x}}_3\}$ of $\langle F\rangle$, {\small
\begin{center}
$\textbf{\emph{x}}_1=(1-ix-x^2+(1+i)x^3, -1+(2-i)x^2+(-1-i)x^3, -i-x+(1+i)x^2, 1+ix+(-1+i)x^2)$,

$\textbf{\emph{x}}_2=(ix+(-1-i)x^3, -ix+(-1+i)x^2+(1+i)x^3, i+(-1-i)x^2, -1+(1-i)x^2)$,

$\textbf{\emph{x}}_3=(-x^3+x^2, ix+(-1+i)x^2+x^3, -x^2+x, -1-ix-ix^2)$.
\end{center}}
\end{example}

Next we present a second illustration of the constructive algorithm.

\begin{example}
Let $M_4(A)$, where $A:=K[x,\sigma,\delta]$, $K:=\mathbb{Q}(t)$, $\sigma:=id_{\mathbb{Q}(t)}$ and
$\delta:=\frac{d}{dt}$; we consider the idempotent matrix $F:=[F^{(1)} \ F^{(2)} \ F^{(3)}\
F^{(4)}]$, $F^{(i)}$ the $ith$ column of $F$, where
\[
F^{(1)}=:
\begin{bmatrix}
2+2t+(13t^2-5t)x+(8t^3-6t^2)x^2+t^3(t-1)x^3\\
2t^2+t+(13t^3-8t^2)x+(8t^4-7t^3)x^2+t^4(t-1)x^3\\
3t+2+(14t^2-8t)x+(8t^3-7t^2)x^2+t^3(t-1)x^3\\
t^2+t+(t^3+6t^2)x+6t^3x^2+t^4x^3
\end{bmatrix},
\]
\[
F^{(2)}=:
\begin{bmatrix}
-t^3x^3-5t^2x^2-3tx+1\\
t+(-3t^2+2t)x+(-5t^3+t^2)x^2-t^4x^3\\
-t^3x^3-5t^2x^2-3tx+1\\
-t^3x^3-5t^2x^2-3tx+1
\end{bmatrix},
\]
\[
F^{(3)}=:
\begin{bmatrix}
t^3x^3+5t^2x^2+3tx-1\\
t^4x^3+5t^3x^2+2t^2x-2t\\
-t-1+(-t^2+5t)x+6t^2x^2+t^3x^3\\
-t^2+t+(-t^3+6t^2)x+2t^3x^2
\end{bmatrix},
\]
\[
F^{(4)}=:
\begin{bmatrix}
0\\
tx\\
tx\\
1+(t^2-2t)x-t^2x^2
\end{bmatrix}.
\]
Applying the algorithm we obtain
\[
U^{(1)}=:
\begin{bmatrix}
2t+1+(10t^2-5t)x+(7t^3-6t^2)x^2+(t^4-t^3)x^3\\
-3t-2+(-14t^2+8t)x+(-8t^3+7t^2)x^2+(-t^4+t^3)x^3\\
-2t+2-t(t-1)x\\
-2t^2+7t-2-t(4t^2-21t+10)x-t^2(t^2-10t+7)x^2+t^3(t-1)x^3
\end{bmatrix},
\]
\[
U^{(2)}=:
\begin{bmatrix}
-t^3x^3-4t^2x^2-tx\\
t^3x^3+5t^2x^2+3tx-1\\
tx+1\\
2t(t-3)x+t^2(t-6)x^2-t^3x^3
\end{bmatrix},
\]
\[
U^{(3)}=:
\begin{bmatrix}
-t-1+(-t^2+3t)x+5t^2x^2+t^3x^3\\
t+2+(t^2-5t)x-6t^2x^2-t^3x^3\\
-tx-1\\
-t+1-t(2t-7)x-t^2(t-6)x^2+t^3x^3
\end{bmatrix},
\]
\[
U^{(4)}=:
\begin{bmatrix}
tx\\
-tx\\
0\\
1
\end{bmatrix};
\]
\[
(U^{-1})^{(1)}=:
\begin{bmatrix}
tx+1\\
t-2+t(t-1)x\\
0\\
-t+2-t(t-4)x+t^2x^2
\end{bmatrix},
\]
\[
(U^{-1})^{(2)}=:
\begin{bmatrix}
tx+1\\
t-1+t(t-1)x\\
1\\
1+(-t^2+3t)x+t^2x^2
\end{bmatrix},
\]
{\small
\[
(U^{-1})^{(3)}=:
\begin{bmatrix}
-t^2x^2-2tx+1\\
t+(-4t^2+4t)x+(-2t^3+5t^2)x^2+t^3x^3\\
1+(-2t^2+t)x+(-t^3+4t^2)x^2+t^3x^3\\
1+(-2t^3+8t^2-5t)x+(-t^4+11t^3-18t^2)x^2+(2t^4-9t^3)x^3-t^4x^4
\end{bmatrix},
\]}
\[
(U^{-1})^{(4)}=:
\begin{bmatrix}
0\\
tx\\
tx\\
1+(t^2-2t)x-t^2x^2
\end{bmatrix}.
\]
With these computations we have
\[
UFU^{-1}=
\begin{bmatrix}
0 &  0 &  0 &  0\\
0 &  0 &  0 &  0\\
0 &  0 &  1 &  0\\
0 &  0 &  0 &  1
\end{bmatrix},
\]
thus, $r=2$ and a base of $\langle F\rangle$ is $X=\{\emph{\textbf{x}}_1,\textbf{\emph{x}}_2\}$,
with
\begin{center}
$\emph{\textbf{x}}_1=(-2t+2-t(t-1)x, tx+1, -tx-1, 0)$,

{\tiny $\textbf{\emph{x}}_2=(-2t^2+7t-2-t(4t^2-21t+10)x-t^2(t^2-10t+7)x^2+t^3(t-1)x^3,
2t(t-3)x+t^2(t-6)x^2-t^3x^3, -t+1-t(2t-7)x-t^2(t-6)x^2+t^3x^3, 1)$.}
\end{center}
\end{example}

{\footnotesize
\begin{center}
\fbox{\parbox[c]{12cm}{
\begin{center}
{\rm \textbf{Algorithm for the Quillen-Suslin theorem:\\ Computational version}}
\end{center}
{\rm \textbf{REQUIRE}:} $A:=K[x;\sigma,\delta]$ and an idempotent matrix $F\in M_s(A)$.
\begin{enumerate}
\item[1:]$k:=0$, $F':=F$;
\item[2:]\textbf{WHILE} $k<s-1$ \textbf{DO}
\item[3:]\quad $k:=k+1$
\item[4:]\quad \textbf{IF} $\max\{\deg(f'_{ij})\mid i=1$ or $j=1\}=-\infty$ \textbf{THEN}
\item[5:]\quad \quad $F':=SubMatrix(F',2..s,2..s)$;
\item[6:]\quad \textbf{ELSE}
\item[7:]\quad \quad (B):
\item[8:]\quad \quad \textbf{IF} $f'_{11}=0$ \textbf{THEN}
\item[9:]\quad \quad \quad if ($f'_{1k}\ne0$) $F':=T_{k1}(-1)F'T_{k1}(-1)^{-1}$;
\item[] \quad \quad \quad if ($f'_{k1}\ne0$) $F':=T_{1k}(-1)F'T_{1k}(-1)^{-1}$;
\item[10:]\quad \quad \textbf{END IF}
\item[11:]\quad \quad (B1):
\item[12:]\quad \quad \textbf{IF} $f'_{11}\in K-\{0\}$ \textbf{THEN}
\item[13:]\quad \quad \quad Apply: OrderReduction1;
\item[14:]\quad \quad \textbf{ELSE}
\item[15:]\quad \quad \quad Apply: (B2) OrderReduction2;
\item[16:]\quad \quad \textbf{END IF}
\item[17:]\quad \textbf{END IF}
\item[18:]\textbf{END WHILE}
\item[19:]\textbf{RETURN} Matrices $U,U^{-1},UFU^{-1}$; a basis $X$ of $\langle F\rangle$;
process step by step.
\end{enumerate}
}}
\end{center}}

\begin{example}
In this example we will illustrate the computational version of the Quillen-Suslin algorithm; let
$M_3(A)$, where $A:=K[x,\sigma,\delta]$, $K:=\mathbb{Q}(t)$,
$\sigma(\frac{p(t)}{q(t)}):=\frac{p(t-1)}{q(t-1)}$ and $\delta:=0$; we have the idempotent matrix
\[
F^{}=
\begin{bmatrix}
1-\frac{2t}{1+t}x & 2t-\frac{2t(3+2t)}{1+t}x & \frac{2t}{(1+t)^2}x\\
\frac{1}{1+t}x & \frac{3+2t}{1+t}x & \frac{-1}{(1+t)^2}x\\
\frac{t}{1+t}x & -t+\frac{t(3+2t)}{1+t}x & 1-\frac{t}{(1+t)^2}x
\end{bmatrix}.
\]
Let $F':=F$, along the example, we will replace the matrices $F'$, $U$ and $U^{-1}$ for the new versions given by
the procedures of the algorithm.

\textit{Step 1}. Since $f'_{11}=1-\frac{2t}{1+t}x$, we will apply the reduction procedure of (B2),
i.e, OrderReduction2:

\textit{Step1.1}: The idea is to convert $f'_{1,i}=0$ for $i>2$ and $f'_{1,2}\neq 0$.

Applying first $T_{2,3}(\frac{-1}{t(1+2t)})$, then $T_{3,2}(t(1+2t)-\frac{t(3+2t)(1+2t)}{1+t}x)$,
and finally permuting the rows and columns $2$ and $3$, we get
\[
UFU^{-1}=
\begin{bmatrix}
1-\frac{2t}{1+t}x & \frac{2}{1+2t} & 0\\
\frac{t(1+2t)}{1+t}x-\frac{2t(1+2t)}{t+2}x^2 & \frac{2t(1+2t)}{(3+2t)(1+t)}x & 0\\
\frac{2t}{(1+2t)(1+t)}x & \frac{-2}{(1+2t)^2} & 1
\end{bmatrix},
\]
where
\[
U =
\begin{bmatrix}
1 & 0 & 0\\
0 & t(1+2t)-\frac{t(3+2t)(1+2t)}{1+t}x & \frac{t(1+2t)}{(1+t)^2}x\\
0 & 1 & \frac{-1}{t(1+2t)}
\end{bmatrix},
\]
\[
U^{-1}=
\begin{bmatrix}
1 & 0 & 0\\
0 & \frac{1}{t(1+2t)} & \frac{3+2t}{1+t}x\\
0 & 1 & -t(1+2t)+\frac{t(3+2t)(1+2t)}{1+t}x
\end{bmatrix}.
\]
\textit{Step 1.2}. Since the new $F'$ is
\[
F'=
\begin{bmatrix}
1-\frac{2t}{1+t}x & \frac{2}{1+2t} & 0\\
\frac{t(1+2t)}{1+t}x-\frac{2t(1+2t)}{t+2}x^2 & \frac{2t(1+2t)}{(3+2t)(1+t)} & 0\\
\frac{2t}{(1+2t)(1+t)}x & \frac{-2}{(1+2t)^2} & 1
\end{bmatrix},
\]
we want to reduce the degree of $f'_{1,1}$; for this we apply $T_{2,1}(\frac{-t(1+2t)}{(1+t)}x)$ and
we obtain
\[
UFU^{-1}=
\begin{bmatrix}
1 & \frac{2}{1+2t}& 0\\
0 & 0 & 0\\
0 & \frac{-2}{(1+2t)^2} & 1
\end{bmatrix},
\]
where the new $U$ and $U^{-1}$ are
\[
U=
\begin{bmatrix}
1& 0& 0\\
\frac{-t(1+2t)}{1+t}x & t(1+2t)-\frac{t(3+2t)(1+2t)}{1+t}x & \frac{t(1+2t)}{(1+t)^2}x\\
0& 1& \frac{-1}{t(1+2t)}
\end{bmatrix},
\]
\[
U^{-1}=
\begin{bmatrix}
1& 0& 0\\
\frac{x}{1+t} & \frac{1}{t(1+2t)} & \frac{3+2t}{1+t}x\\
\frac{t(1+2t)}{1+t}x & 1 & -t(1+2t)+\frac{t(3+2t)(1+2t)}{1+t}x
\end{bmatrix}.
\]
\textit{Step 2}. The new $F'$ is
\[
F' =
\begin{bmatrix}
1 & \frac{2}{1+2t} & 0\\
0 & 0 & 0\\
0 & \frac{-2}{(1+2t)^2} & 1
\end{bmatrix};
\]
since $f'_{1,1}=1$ we apply (B1), i.e., OrderReduction1, for this we consider the matrices
\[
S=
\begin{bmatrix}
1 & \frac{2}{1+2t} & 0\\
0 & 1 & 0\\
0 & 0 & 1
\end{bmatrix},
\text{ and }
S^{-1}=
\begin{bmatrix}
1 & \frac{-2}{1+2t} & 0\\
0 & 1 & 0\\
0 & 0 & 1
\end{bmatrix},
\]
and then
\[
S\,F'S^{-1}=
\begin{bmatrix}
1& 0& 0\\
0& 0& 0\\
0& \frac{-2}{(1+2t)^2} & 1
\end{bmatrix}.
\]
Therefore, the new $F'$ is
\[
F'=\begin{bmatrix}
0 & 0\\
\frac{-2}{(1+2t)^2} & 1
\end{bmatrix},
\text{ and }
UFU^{-1}=
\begin{bmatrix}
1& 0& 0\\
0& 0& 0\\
0& \frac{-2}{(1+2t)^2} & 1
\end{bmatrix},
\]
where the new $U$ and $U^{-1}$ are
\[
U=
\begin{bmatrix}
1-\frac{2t}{1+t}x & 2t-\frac{2t(3+2t)}{1+t}x & \frac{2t}{(1+t)^2}x\\
\frac{-t(1+2t)}{1+t}x & t(1+2t)-\frac{t(3+2t)(1+2t)}{1+t}x & \frac{t(1+2t)}{(1+t)^2}x\\
0 & 1 & \frac{-1}{t(1+2t)}
\end{bmatrix},
\]
\[
U^{-1}=
\begin{bmatrix}
1 & \frac{-2}{1+2t} & 0\\
\frac{x}{1+t} & \frac{1}{t(1+2t)}-\frac{2}{(1+t)(3+2t)}x & \frac{3+2t}{1+t}x\\
\frac{t(1+2t)}{1+t}x & 1-\frac{2t(1+2t)}{(1+t)(3+2t)}x & -t(1+2t)+\frac{t(3+2t)(1+2t)}{1+t}x
\end{bmatrix}.
\]
Since $f'_{1,1}=0$, we apply $T_{1,2}(-1)$, we get
\[
UFU^{-1}=
\begin{bmatrix}
1& 0& 0\\
0& \frac{2}{(1+2t)^2} & \frac{-4t^2-4t+1}{(1+2t)^2}\\
0& \frac{-2}{(1+2t)^2} & \frac{4t^2+4t-1}{(1+2t)^2}
\end{bmatrix}\text{ and }
F'=
\begin{bmatrix}
\frac{2}{(1+2t)^2} & \frac{-4t^2-4t+1}{(1+2t)^2}\\
\frac{-2}{(1+2t)^2}& \frac{4t^2+4t-1}{(1+2t)^2}
\end{bmatrix},
\]
where the new $U$ and $U^{-1}$ are
\[
U=
\begin{bmatrix}
1-\frac{2t}{1+t}x & 2t-\frac{2t(3+2t)}{1+t}x & \frac{2t}{(1+t)^2}x\\
\frac{-t(1+2t)}{1+t}x & 2t^2+t-1-\frac{t(3+2t)(1+2t)}{1+t}x & \frac{1}{t(1+2t)}+\frac{t(1+2t)}{(1+t)^2}x\\
0 & 1 & \frac{-1}{t(1+2t)}
\end{bmatrix},
\]
\[
U^{-1}=
\begin{bmatrix}
1 & \frac{-2}{1+2t} & \frac{-2}{1+2t}\\
\frac{1}{1+t}x & \frac{1}{t(1+2t)}-\frac{2}{(1+t)(3+2t)}x & \frac{1}{t(1+2t)}+\frac{4t^2+12t+7}{(1+t)(3+2t)}x\\
\frac{t(1+2t)}{1+t}x & 1-\frac{2t(1+2t)}{(1+t)(3+2t)}x & -2t^2-t+1+\frac{t(1+2t)(4t^2+12t+7)}{(1+t)(3+2t)}x
\end{bmatrix}.
\]
Since $f'_{1,1}=\frac{2}{(1+2t)^2}$ is invertible, we apply OrderReduction1 with matrices
\[
T=
\begin{bmatrix}
1& -2t^2-2t+\frac{1}{2}\\
1& 1
\end{bmatrix}
\text{ and }
T^{-1} =
\begin{bmatrix}
\frac{2}{(1+2t)^2} & \frac{4t^2+4t-1}{(1+2t)^2}\\
\frac{-2}{(1+2t)^2} & \frac{2}{(1+2t)^2}
\end{bmatrix},
\]
so
\[
T\,F'T^{-1}=
\begin{bmatrix}
1 & 0& 0\\
0 & 1& 0\\
0 & 0& 0
\end{bmatrix}.
\]
Thus, the new $F'$ is
\[
F'=
\begin{bmatrix}
0
\end{bmatrix}
\text{ and }
UFU^{-1}=
\begin{bmatrix}
1 & 0& 0\\
0 & 1& 0\\
0 & 0& 0
\end{bmatrix},
\]
where the new $U$ and $U^{-1}$ are
\[
U=
\begin{bmatrix}
1-\frac{2t}{1+t}x & 2t-\frac{2t(3+2t)}{1+t}x & \frac{2t}{(1+t)^2}x\\
\frac{-t(1+2t)}{1+t}x & -t-\frac{1}{2}-\frac{t(3+2t)(1+2t)}{1+t}x & \frac{1+2t}{2t}+\frac{t(1+2t)}{(1+t)^2}x\\
\frac{-t(1+2t)}{1+t}x& 2t^2+t-\frac{t(3+2t)(1+2t)}{1+t}x & \frac{t(1+2t)}{(1+t)^2}x
\end{bmatrix},
\]
\[
U^{-1}=
\begin{bmatrix}
1& 0& \frac{-2}{1+2t}\\
\frac{x}{1+t} & \frac{-2}{(1+t)(3+2t)}x & \frac{1}{t(1+2t)}\\
\frac{t(1+2t)}{1+t}x & \frac{2t}{1+2t}-\frac{2t(1+2t)}{(1+t)(3+2t)}x & \frac{1}{1+2t}
\end{bmatrix}.
\]
Permuting, we have finally
\[
UFU^{-1}=
\begin{bmatrix}
0& 0& 0\\
0& 1& 0\\
0& 0& 1
\end{bmatrix},
\]
where the new $U$ and $U^{-1}$ are
\[
U=
\begin{bmatrix}
\frac{-t(1+2t)}{1+t}x & 2t^2+t-\frac{t(3+2t)(1+2t)}{1+t}x & \frac{t(1+2t)}{(1+t)^2}x\\
1-\frac{2t}{1+t}x & 2t-\frac{2t(3+2t)}{1+t}x & \frac{2t}{(1+t)^2}x\\
\frac{-t(1+2t)}{1+t}x & -t-\frac{1}{2}-\frac{t(3+2t)(1+2t)}{1+t}x & \frac{1+2t}{2t}+\frac{t(1+2t)}{(1+t)^2}x
\end{bmatrix},
\]
\[
U^{-1}=
\begin{bmatrix}
\frac{-2}{1+2t} & 1 & 0\\
\frac{1}{t(1+2t)} & \frac{1}{1+t}x & \frac{-2}{(1+t)(3+2t)}x\\
\frac{1}{1+2t} & \frac{t(1+2t)}{1+t}x & \frac{2t}{1+2t}-\frac{2t(1+2t)}{(1+t)(3+2t)}x
\end{bmatrix}.
\]
Therefore, $r=2$ and the last two rows of $U$ conform a basis
$X=\{\emph{\textbf{x}}_1,\textbf{\emph{x}}_2\}$, of $\langle F\rangle$,
\begin{center}
$\emph{\textbf{x}}_1=(1-\tfrac{2t}{1+t}x,2t-\tfrac{2t(3+2t)}{1+t}x,\tfrac{2t}{(1+t)^2}x)$,
$\emph{\textbf{x}}_2=(\tfrac{-t(1+2t)}{1+t}x,
-t-\tfrac{1}{2}-\tfrac{t(3+2t)(1+2t)}{1+t}x,\tfrac{1+2t}{2t}+\tfrac{t(1+2t)}{(1+t)^2}x)$.
\end{center}

\end{example}

\begin{example}
Let $M_4(A)$, where $A:=K[x,\sigma,\delta]$, $K:=\mathbb{Q}(t)$, $\sigma(f(t)):=f(qt)$ and
$\delta(f(t)):=\tfrac{f(qt)-f(t)}{t(q-1)}$, where $q\in K-\{0,1\}$; we consider the idempotent
matrix $F:=[F^{(1)} \ F^{(2)} \ F^{(3)}\ F^{(4)}]$, $F^{(i)}$ the $ith$ column of $F$ and
$a\in\mathbb{Q}$, where

\[
F^{(1)}=
\begin{bmatrix}
-t^{2}qx^2\\
\left(-ta+2\,t \right)x-2\,a+2\\
tx+2\\
-1
\end{bmatrix},
\]

\[
F^{(2)}=
\begin{bmatrix}
-2\,tx+2\\
-{t}^{2}q{x}^{2}+ \left( ta-4\,t \right) x+2\,a-1\\
-tx-2\\
tx+2
\end{bmatrix},
\]

\[
F^{(3)}=
\begin{bmatrix}
-tx-2\\
\left( -2\,{t}^{2}qa+3\,{t}^{2}q \right) {x}^{2}+ \left( {a}^{2}t-8\,ta+8\,t \right) x+2\,{a}^{2}-3\,a+1\\
{t}^{2}q{x}^{2}+ \left( -ta+4\,t \right) x-2\,a+2\\
\left( ta-2\,t \right) x+2\,a-2
\end{bmatrix},
\]

\[
F^{(4)}=
\begin{bmatrix}
-{t}^{3}{q}^{3}{x}^{3}+\left( -{q}^{2}{t}^{2}-5\,{t}^{2}q \right) {x}^{2}-5\,tx+2\\
-{t}^{3}{q}^{3}{x}^{3}+\left( -{q}^{2}{t}^{2}-3\,{t}^{2}q \right) {x}^{2}+ \left( -ta+t \right) x-2\,a+2\\
tx+2\\
{t}^{2}q{x}^{2}+2\,tx-1
\end{bmatrix}.
\]
Applying the algorithm we obtain
\[
U =
\begin{bmatrix}
tx+1 & 0 & {t}^{2}q{x}^{2}+2\,tx-1 & {t}^{2}q{x}^{2}+3\,tx\\
1 & -tx-2 & \left( -ta+2\,t \right) x-2\,a+2 & -{t}^{2}q{x}^{2}-2\,tx+2\\
tx-1 & 1 & {t}^{2}q{x}^{2}+a-1 & {t}^{2}q{x}^{2}+2\,tx-1\\
1 & 0 & tx & tx+1
\end{bmatrix},
\]
{\tiny
\[
U^{-1}=
\begin{bmatrix}
tx & -1 & -tx-2 & 0\\
a-1 & -tx+a-1 & -{t}^{2}q{x}^{2}+ \left( ta-4\,t \right) x+2\,a-1 & {t}^{3}{q}^{3}{x}^{3}-
\left( -q+a-4 \right) {t}^{2}q{x}^{2}+ \left( -3\,ta+3\,t \right) x+1\\
-1 & -1 & -tx-2 & {t}^{2}q{x}^{2}+3\,tx\\
0 & 1 & tx+2 & -{t}^{2}q{x}^{2}-2\,tx+1
\end{bmatrix},
\]}
\[
UFU^{-1} =
\begin{bmatrix}
0 & 0 & 0 & 0\\
0 & 0 & 0 & 0\\
0 & 0 & 1 & 0\\
0 & 0 & 0 & 1
\end{bmatrix},
\]

Therefore, $r=2$ and the last two rows of $U$ conform a basis
$X=\{\emph{\textbf{x}}_1,\textbf{\emph{x}}_2\}$, of $\langle F\rangle$,
\begin{center}
$\emph{\textbf{x}}_1=(tx-1,1,{t}^{2}q{x}^{2}+a-1,{t}^{2}q{x}^{2}+2\,tx-1)$,
$\emph{\textbf{x}}_2=(1,0,tx,tx+1)$.
\end{center}

\end{example}

\section{Some remarks about the implementation}

\noindent In this final section we present some comments about the implementation of the
computational version of the Quillen-Suslin algorithm.

\begin{remark}
The OrderReduction1 is based in the implementation of the procedure $(B1)$ in the proof of Theorem
\ref{theorem9.6.1}; for the OrderReduction2, the following algorithm describes its functionality:
{\footnotesize
\begin{center}
\fbox{\parbox[c]{12cm}{
\begin{center}
{\rm \textbf{Algorithm OrderReduction2}}
\end{center}
{\rm \textbf{REQUIRE}:} $A:=K[x;\sigma,\delta]$ and an idempotent matrix $F\in M_s(A)$ with
$\deg(f_{11})\ge1$.
\begin{enumerate}
\item[1:]\quad Make $f_{1,j}=0$ for $j>2$ and $f_{1,2}\ne0$;
\item[2:]\quad Reduce degree of $f_{1,1}$;
\item[3:]\quad \textbf{IF} {$f_{1,1}=0$}
\item[4:]\quad\quad \textbf{IF} {$\max\{\deg(f_{i,j})>0\mid i=1$ or $j=1\}>0$}
\item[5:]\quad\quad\quad Make $f_{1,j}=0$ for $j>2$ and $f_{1,2}\ne0$;
\item[6:]\quad\quad\quad $F:=P_{12}\,F\,P_{12}$;
\item[7:]\quad\quad\quad Apply: OrderReduction1;
\item[8:]\quad\quad \textbf{ELSE}
\item[9:]\quad\quad\quad $F':=SubMatrix(F,2..s,2..s)$;
\item[10:]\quad\quad \textbf{ENDIF}
\item[11:]\quad \textbf{ELSE}
\item[12:]\quad\quad Apply: OrderReduction1;
\item[13:]\quad \textbf{ENDIF}
\item[14:]\textbf{RETURN} Matrices $U,U^{-1}, F'$ and $UFU^{-1}=\begin{bmatrix}\alpha & 0\\
0 & F' \end{bmatrix}$, with $\alpha\in\{0,1\}$.
\end{enumerate}
}}
\end{center}}
\end{remark}

\begin{remark}\label{remark4.2}
For the implementation of the Quillen-Suslin algorithm we used ${\rm Maple}^\circledR$ 2016, and we
create a library called \texttt{OrePolyToolKit.lib} consisting in two packages:
\begin{itemize}
\item \textbf{\texttt{OrePolyUtility}}: This is a new useful collection of functions for operating matrices, vectors
and lists over an \texttt{UnivariateOreRing} $K[x;\sigma,\delta]$; the \texttt{UnivariateOreRing}
structure was taken from the library \texttt{OreTools} within the standard Maple libraries.
\item \textbf{\texttt{OrePolyQS}}: This is the most important new collection of functions related to the Quillen-Suslin
algorithm over $K[x;\sigma,\delta]$; the main routine of the algorithm was implemented here, the
following functions of this package are fundamentals:
\begin{itemize}
\item \textbf{\texttt{GenerateIdemp}}: This function generates idempotent matrices over $K[x;\sigma,\delta]$, the arguments
are the matrix order and the \texttt{UnivariateOreRing}, and return an idempotent matrix of the
given dimension over the respective \texttt{UnivariateOreRing}.
\item \textbf{\texttt{QSAlgKsd}}: This is the main function of the algorithm, it shows the sequence
of all steps of the Quillen-Suslin algorithm presented in this paper; the arguments are the
idempotent matrix and the \texttt{UnivariateOreRing}, and return the matrix $UFU^{-1}$ in the form
of Theorem \ref{theorem9.6.1}, the matrices $U$ and $U^{-1}$, the basis of $\langle F\rangle$ and
the complete process step by step.
\end{itemize}
\end{itemize}
\end{remark}





\begin{thebibliography}{200}

\bibitem{Artamonov2}\textbf{Artamonov, V.},  \textit{Serre's quantum problem},
Russian Math. Surveys, 53(4), 1998, 657-730.

\bibitem{Artamonov3}\textbf{Artamonov, V.},  \textit{On projective modules over quantum polynomials},
Journal of Mathematical Sciences, 93(2), 1999, 135-148.

\bibitem{Bass1} \textbf{Bass, H.}, \textit{Proyective modules over algebras},
Annals of Math. 73, 532-542, 1962.

\bibitem{Cohn1}\textbf{Cohn, P.}, \textit{Free Ideal Rings and Localizations in General Rings}, Cambridge University Press, 2006.

\bibitem{Gallego}\textbf{Gallego, C. and Lezama, O.}, \textit{Projective modules and Gröbner bases for skew $PBW$ extensions}, Dissertationes Mathematicae,
521, 2017, 1-50.

\bibitem{Lam}\textbf{Lam, T.Y.}, \textit{Serre's Problem on Projective Modules
}, Springer Monographs in Mathematics, Springer, 2006.

\bibitem{Quillen}\textbf{Quillen, D.}, \textit{Proyective modules over polynomial rings},
Invent. Math., 36, 1976, 167-171.

\bibitem{Suslin1}\textbf{Suslin, A.A.}, \textit{Proyective modules over polynomial rings are free},
Soviet Math. Dokl., 17, 1976, 1160-1164.

\end{thebibliography}
\end{document}